\documentclass[12pt,reqno]{amsart}
\usepackage{graphicx}
\usepackage{color}
\usepackage{amsmath,amssymb,amsthm,amsfonts}
\usepackage{graphicx}
\usepackage{color}

\usepackage{mathptmx}
\usepackage{multicol}
\def\R{\mathbb{R}}
\makeatletter

\newcommand{\Rmnum}[1]{\expandafter\@slowromancap\romannumeral #1@}
\makeatother

\newtheorem{thm}{Theorem}[section]
\newcommand{\norm}[1]{\left\lVert#1\right\rVert}

\newtheorem{lemma}[thm]{Lemma}
\newtheorem{remark}{Remark}[section]
\newtheorem{theorem}[thm]{Theorem}

\usepackage[numbers,sort&compress]{natbib}
\begin{document}
\author{Hai-Yang Jin}
\address{Department of Mathematics, South China University of Technology, Guangzhou 510640, China}
\email{mahyjin@scut.edu.cn}

\author{Zhi-An Wang}
\address{Department of Applied Mathematics, Hong Kong Polytechnic University, Hung Hom,
 Hong Kong}
\email{mawza@polyu.edu.hk}

\title[Global stabilization for full ARKS model ]{Global stabilization of the full attraction-repulsion Keller-Segel system }

\begin{abstract}
We are concerned with the  following full Attraction-Repulsion Keller-Segel (ARKS) system
\begin{equation}\label{ARKS}\tag{$\ast$}
\begin{cases}
u_t=\Delta u-\nabla\cdot(\chi u\nabla v)+\nabla\cdot(\xi u\nabla w), &x\in \Omega, ~~t>0,\\
  v_t=D_1\Delta v+\alpha u-\beta v,& x\in \Omega, ~~t>0,\\
 w_t=D_2\Delta w+\gamma u-\delta w, &x\in \Omega, ~~t>0,\\
u(x,0)=u_0(x),~v(x,0)= v_0(x),  w(x,0)= w_0(x)  & x\in \Omega,
\end{cases}
\end{equation}
in a bounded domain $\Omega\subset \R^2$ with smooth boundary subject to homogeneous Neumann boundary conditions.
%The parameters $D_1,D_2,\chi,\xi,\alpha,\beta,\gamma$ and $\delta$ are positive.
By constructing an appropriate Lyapunov functions, we establish the boundedness and asymptotical behavior of solutions to the system \eqref{ARKS} with large initial data.  Precisely,  we show that if the parameters satisfy $\frac{\xi\gamma}{\chi\alpha}\geq \max\Big\{\frac{D_1}{D_2},\frac{D_2}{D_1},\frac{\beta}{\delta},\frac{\delta}{\beta}\Big\}$ for all positive parameters $D_1,D_2,\chi,\xi,\alpha,\beta,\gamma$ and $\delta$,  the system \eqref{ARKS} has a unique global classical solution $(u,v,w)$, which converges to the constant steady state  $(\bar{u}_0,\frac{\alpha}{\beta}\bar{u}_0,\frac{\gamma}{\delta}\bar{u}_0)$  as $t\to+\infty$, where $\bar{u}_0=\frac{1}{|\Omega|}\int_\Omega u_0dx$. Furthermore, the decay rate is exponential if  $\frac{\xi\gamma}{\chi\alpha}> \max\Big\{\frac{\beta}{\delta},\frac{\delta}{\beta}\Big\}$. This paper provides the first results on the full ARKS system with unequal chemical diffusion rates (i.e. $D_1\ne D_2$) in multi-dimensions.
  \end{abstract}

\subjclass[2000]{35A01, 35B40, 35B44, 35K57, 35Q92, 92C17}

\keywords{Chemotaxis, attraction-repulsion, global stability, exponential decay rate}

   % acknowledge support, etc
   %\thanks{This research was partially supported by NSF grant DOA-123456789.}
   %\thanks{We would like to thank our colleagues for their helpful criticism.}

   % dedication
   %\dedicatory{Dedicated to Professor}

   % today's date, or fill in whatever date you prefer
   %\date{\today}

% This ends the top matter information.
% We can now tell LaTeX to display the top matter.
\maketitle

\numberwithin{equation}{section}
\section{Introduction}
To  describe the aggregation of {\it Microglia} in the central nervous system  in Alzhemer's disease due to the interaction of chemoattractant (i.e. $\beta$-amyloid) and chemorepellent (i.e. TNF-$\alpha$), Luca  et al. \cite{M-A-L-A}  proposed the  following  attraction-repulsion chemotaxis system
\begin{equation}\label{1-1*}
\begin{cases}
u_t=\Delta u-\nabla\cdot(\chi u\nabla v)+\nabla\cdot(\xi u\nabla w), &x\in \Omega, ~~t>0,\\
 \tau_1 v_t=D_1\Delta v+\alpha u-\beta v,& x\in \Omega, ~~t>0,\\
\tau_2 w_t=D_2\Delta w+\gamma u-\delta w, &x\in \Omega, ~~t>0,\\
\frac{\partial u}{\partial \nu}=\frac{\partial v}{\partial \nu}=\frac{\partial w}{\partial \nu}=0,&x\in \partial\Omega, ~~t>0,\\
u(x,0)=u_0(x),~\tau_1v(x,0)=\tau_1v_0(x),~\tau_2w(x,0)=\tau_2w_0(x),& x\in \Omega,
\end{cases}
\end{equation}
where $\Omega$ is a bounded domain in $\R^n$  with smooth boundary $\partial\Omega$ and $\nu$ denotes the outward normal vector of $\partial \Omega$. The density of {\it Microglia} cells is denoted by $u(x,t)$, while $v(x,t)$ and $w(x,t)$ denote the concentration of chemoattractant and chemorepllent, respectively. The model (\ref{1-1*}) can also be regarded as a particularized model proposed in \cite{P-H-2002} to model the quorum sensing effect in chemotaxis.

When $\xi=0$, the variable $w$ can be decoupled from the  system \eqref{1-1*}, where the variables $u$ and $v$ satisfy the classical  attractive Keller-Segel (KS) system
\begin{equation}\label{AL}
\begin{cases}
 u_{t}=\Delta u-\chi\nabla\cdot(u\nabla v), &x\in\Omega,\ t>0,\\
  \tau_1v_{t}=D_1\Delta v+\alpha u-\beta v, &x\in\Omega,\
 t>0.\\
 \end{cases}
\end{equation}
The KS model (\ref{AL}) has been extensively studied in the past four decades in various perspectives and massive results are available (cf.  survey articles \cite{Horstmann-D, BBTW-2015} and references therein). One of the mostly studied topics for the KS model \eqref{AL} is the boundedness and blowup of solutions in two or higher dimensions \cite{Nagai-Funk, W-JMPA-2013,horstmann2001blow} based on the following Lyapunov function:
\begin{equation*}\label{F1}
\mathcal{E}_1(u,v)=\int_\Omega u\ln u -\chi \int_\Omega uv+\frac{\beta\chi}{2\alpha}\int_\Omega v^2+\frac{\chi D_1}{2\alpha}\int_\Omega |\nabla v|^2.
\end{equation*}
If $\chi=0$, the variable $v$ can be decoupled and $(u,w)$ satisfies the following repulsive Keller-Segel model
\begin{equation}\label{RL}
\begin{cases}
u_t=\Delta u+\xi\nabla\cdot( u\nabla w), &x\in \Omega, ~~t>0,\\
\tau_2w_t=D_2\Delta w+\gamma u-\delta w, &x\in \Omega, ~~t>0.\\
\end{cases}
\end{equation}
Compared to the attractive KS model (\ref{AL}), the results on the repulsive KS model (\ref{RL}) are much less. The global existence of classical solutions in two dimensions and weak solutions in three or four dimensions were established in \cite{CLM-2008} based on the following Lyapunov function
\begin{equation*}\label{F2}
\mathcal{E}_2(u,w)=\int_\Omega u\ln u +\frac{\tau_2\xi}{2\gamma}\int_\Omega |\nabla w|^2
\end{equation*}
which is difference from the one for the attractive KS model. A further investigation on the  repulsive KS model was made in \cite{T-DCDSB-2013}.

Roughly speaking,  the attraction-repulsion Keller-Segel model \eqref{1-1*} can be regarded as a superposition of the attractive and repulsive KS models. Hence one may expect the ARKS model should behave more or less the same as the attractive or repulsive models. However it is not straightforward to justify this suspicion due to the interaction between attraction and repulsion. In particular, as we recalled above, the understanding of the attractive and repulsive KS models heavily rely on the finding of Lyapunov functions.  Therefore to have a comprehensive understanding for the ARKS model, finding appropriate Lyapunov function is indispensable.  This is by no means an easy work for a strongly coupled cross-diffusion system of PDEs like ARKS model. A sequence of works thus have been stimulated to reveal the mystery underlying the model gradually. The first such progress was made by Tao and Wang \cite{TW-M3AS-2013} who found that the solution behavior of the ARKS model was essentially determined by  the sign of
\begin{equation*}
\theta_1=\xi\gamma-\chi\alpha,
\end{equation*}
which describes the competition of attraction and repulsion. More precisely, they showed that if $D_1=D_2=1$ and $\tau_1=\tau_2=0$, the ARKS system \eqref{1-1*} has a unique classical solution with
uniform-in-time bound if $\theta_1\geq 0$ (i.e. repulsion dominates over or cancels attraction) in higher dimensions ($n\geq 2$). The main idea of \cite{TW-M3AS-2013} was a transformation $s=\xi w-\chi v$ which may significantly simplify the system and has become a major source for many of subsequent researches on the ARKS model. For the opposite case $\theta_1<0$ (i.e. attraction dominates over repulsion), it was shown that the solution of system \eqref{1-1} may blow up in finite time if initial mass is large \cite{Espejo2014, LL-NARWA-2016} and exist globally for small initial mass \cite{Espejo2014} in two dimensions. If $\tau_1=1$ and $\tau_2=0$, Jin and Wang \cite{JW-JDE-2016} constructed a Lyapunov function
\begin{equation*}\label{F3}
\begin{aligned}
\mathcal{E}_3(u,v,w)=&\int_\Omega u\ln u -\chi \int_\Omega uv+\frac{\beta\chi}{2\alpha}\int_\Omega v^2\\
&+\frac{\chi D_1}{2\alpha}\int_\Omega |\nabla v|^2+\frac{\xi\delta}{2\gamma}\int_\Omega w^2+\frac{\xi D_2}{2\gamma}\int_\Omega |\nabla w|^2,
\end{aligned}
\end{equation*}
to establish the global existence of uniformly-in-time bounded classical solutions in two dimensions for large initial data if $\theta_1\geq 0$. Conversely if $\theta_1<0$,  they showed there exists a critical mass $m_*$  such that the solution blows up if $\int_\Omega u_0>m_*$ and globally exists if $\int_\Omega u_0<m_*$.

If the three equations of the ARKS model \eqref{1-1*} are all parabolic (i.e. $\tau_1=\tau_2=1$), it is much harder to study and much less results are available.  We recall the known results below. In one dimension, the global existence  of classical {solutions}, non-trivial stationary {state}, asymptotic behavior and pattern formation of the system \eqref{1-1*} have been studied in  \cite{JW-M2AS-2015, LSW-DCDSB-2013,LW-JBD-2012}.  In two dimensions, when $D_1=D_2$, it was shown in \cite{TW-M3AS-2013} that global classical solutions exist for large data if $\beta=\delta$ and for small data if $\beta\ne \delta$ when $\theta_1\geq 0$ (i.e. repulsion dominates over or cancels attraction). Subsequently the global existence of large-data solutions was extended to the case $\beta\ne \delta$ in \cite{J-JMAA-2015,LT-M2AS-2015}. Moreover, for $\beta\neq \delta$, when cell mass is small, it was shown that the global classical solution will exponentially converge to the unique constant steady state  $(\bar{u}_0,\frac{\alpha}{\beta}\bar{u}_0,\frac{\gamma}{\delta}\bar{u}_0)$ with $\bar{u}_0=\frac{1}{|\Omega|}\int_\Omega u_0$ in \cite{LM-NAR-2016,LMW-JMAA-2015}, {\color{black} which was further elaborated  by assuming
\begin{equation}\label{sm}
\bar{u}_0<\frac{4\beta\delta}{\chi\alpha(\beta-\delta)^2}\ \text{and}\
\xi\gamma>\frac{4\beta\delta(\chi\alpha \bar{u}_0+1)}{[4\beta\delta-(\beta-\delta)^2\chi\alpha \bar{u}_0]\bar{u}_0}
\end{equation}
in  \cite{Lin-M3AS-2018} wherein the convergence rate was, however, not given}.  Whether or not  the same results holds for large initial data in multi-dimensions still remains unknown. Part of above-mentioned results have been extended to the multi-dimensional whole space in \cite{JL-AML-2015, SW-JMAA-2015}.  We should underline that all existing results in two or higher dimensions recalled above for the case $\tau_1=\tau_2=1$ are essentially based on the assumption $D_1=D_2$ so that the idea of making a change of variable $s=\xi w-\chi v$ introduced in \cite{TW-M3AS-2013} can be employed. To the best of our knowledge, no result for the case $\tau_1=\tau_2=1$ and $D_1\ne D_2$ has been available to \eqref{1-1*} in multi-dimensions to date. It is the purpose of this paper to exploit this challenging case and contribute some results, where the corresponding  ARKS model \eqref{1-1*} reads as
\begin{equation}\label{1-1}
\begin{cases}
u_t=\Delta u-\nabla\cdot(\chi u\nabla v)+\nabla\cdot(\xi u\nabla w), &x\in \Omega, ~~t>0,\\
  v_t=D_1\Delta v+\alpha u-\beta v,& x\in \Omega, ~~t>0,\\
w_t=D_2\Delta w+\gamma u-\delta w, &x\in \Omega, ~~t>0,\\
\frac{\partial u}{\partial \nu}=\frac{\partial v}{\partial \nu}=\frac{\partial w}{\partial \nu}=0,&x\in \partial\Omega, ~~t>0,\\
u(x,0)=u_0(x),~v(x,0)=v_0(x),~w(x,0)=w_0(x),& x\in \Omega.
\end{cases}
\end{equation}
The main challenge is that {\color{black}when $D_1\neq D_2$} the conventional approach of using the transformation in \cite{TW-M3AS-2013} is no longer effective and new ideas are desirable. Here we shall construct a Lyapunov functional for (\ref{1-1}) which allows us to establish the global boundedness and asymptotic behavior of solutions to (\ref{1-1}) in some parameter regime. Specifically, the following results are obtained in the paper.
\begin{theorem}\label{LT-3}
%{\color{black}Suppose $\theta_1>0$}.
Let $\Omega$ be a bounded domain  in $\R^2$ with smooth boundary. Suppose that $0\leq (u_0,v_0,w_0)\in [W^{1,\infty}(\Omega)]^3$ and  the parameters satisfy
\begin{equation}\label{LT-31}
\frac{\xi\gamma}{\chi\alpha}\geq \max\Big\{\frac{D_1}{D_2},\frac{D_2}{D_1},\frac{\beta}{\delta},\frac{\delta}{\beta}\Big\}.
\end{equation}
 Then  the problem \eqref{1-1} has a unique classical solution $(u,v,w)\in [C^0([0,\infty)\times \bar{\Omega})\cap C^{2,1}((0,\infty)\times\bar{\Omega})]^3$, which satisfies
 \begin{equation}\label{LT-12}
\|u(\cdot,t)\|_{L^\infty}\leq C
\end{equation}
and  converges to the constant steady state $(\bar{u}_0,\frac{\alpha}{\beta}\bar{u}_0,\frac{\gamma}{\delta}\bar{u}_0)$  as $t\to+\infty$, where $\bar{u}_0=\frac{1}{|\Omega|}\int_\Omega u_0dx$ and $C$ is a positive constant.
Furthermore, if $\frac{\xi\gamma}{\chi\alpha}> \max\Big\{\frac{\beta}{\delta},\frac{\delta}{\beta}\Big\}$,  the decay is exponential.
 \end{theorem}
 {\color{black}
\begin{remark}\label{rem1}
\em{
If $D_1=D_2=1$ and $\beta=\delta$, Tao \& Wang \cite[Proposition 2.6]{TW-M3AS-2013} proved that if $\theta_1>0$ the global classical solution $(u,v,w)$ of system \eqref{1-1} exists and exponentially converges to the constant steady state $(\bar{u}_0,\frac{\alpha}{\beta}\bar{u}_0,\frac{\gamma}{\delta}\bar{u}_0)$ as $t\to\infty$. Hence in this paper, we will, unless otherwise mentioned, focus on the case $D_1\ne D_2$ or $\beta\ne \delta$ under which the condition \eqref{LT-31} implies $\theta_1>0$.
}
\end{remark}

 \begin{remark}
 \em{
The results of Theorem \ref{LT-3} hold for all $D_1, D_2, \alpha, \beta, \xi, \gamma>0$ without any smallness conditions on initial data under the neat parameter regime given by (\ref{LT-31}). In the case $D_1=D_2=1$ and $\beta\ne \delta$, the same result was recently obtained in \cite{Lin-M3AS-2018} under the essential assumption (\ref{sm}) where the initial cell mass can not be arbitrarily large and parameter regime depends upon the initial data. Hence our results not only improve those of \cite{Lin-M3AS-2018}, but also cover the case $D_1\ne D_2$ for which no results have been known so far.
}
\end{remark}
%\begin{remark}
%If  $D_1\ne D_2$ or $\beta\ne \delta$, the condition \eqref{LT-31} implies that $\theta_1>0$. Hence  we only need the condition \eqref{LT-31} to show the large time behavior of solution. However, when $\theta_1=0$, and $D_1\ne D_2$ or $\beta\ne \delta$, whether or not the boundedness and  large time behavior results in Theorem \ref{LT-3}  still hold under condition \eqref{LT-31}  are open.
%\end{remark}

%When $D_1=D_2$, it has been proved in \cite[Proposition 2.3]{LSW-DCDSB-2013} that the constant steady state $(\bar{u}_0,\frac{\alpha}{\beta}\bar{u}_0,\frac{\gamma}{\delta}\bar{u}_0)$ is locally asymptotically state under \eqref{LT-31}. Theorem \ref{LT-3} show that the  constant steady state $(\bar{u}_0,\frac{\alpha}{\beta}\bar{u}_0,\frac{\gamma}{\delta}\bar{u}_0)$ is indeed globally asymptotically stable  the without smallness assumptions on the initial data, which  improve the results in  \cite{Lin-M3AS-2018}.
}
\noindent {\bf{Outline of proof}}: We first establish the boundedness criterion of solution for system \eqref{1-1} such  that the boundedness of $\|u\|_{L^\infty}$ can be reduced to prove the boundedness of $\|u\|_{L^p}$ with $p>\max\{1,\frac{n}{2}\}$.  Motivated by the results in \cite{J-JMAA-2015,LT-M2AS-2015},  we know that the boundedness of $\|u\|_{L^2}$ holds in two dimensions if there exists a constant $c_1>0$ such that
 \begin{equation}\label{BC-K}
 \|u\ln u\|_{L^1}+\|\nabla v\|_{L^2}+\|\nabla w\|_{L^2}\leq c_1.
 \end{equation}
 Hence to show the global existence of classical solutions in two dimensions, we only need to prove \eqref{BC-K}.  When $D_1=D_2$ and $\theta_1>0$, using the  transformation $s=\xi w-\chi v$ as in \cite{TW-M3AS-2013}, one can derive  the following entropy inequality (cf. \cite{LT-M2AS-2015,J-JMAA-2015})
 \begin{equation}\label{LY}
 \begin{split}
&\frac{d}{dt}\left(\int_\Omega u\ln u +\frac{1}{2\theta_1}\int_\Omega |\nabla s|^2\right)+\int_\Omega \frac{|\nabla u|^2}{u}+\frac{D_1}{2\theta_1}\int_\Omega |\Delta s|^2+\frac{\delta}{\theta_1}\int_\Omega |\nabla s|^2\leq c_2,
 \end{split}
 \end{equation}
 which can be used to derive \eqref{BC-K} and hence the boundedness of solutions.

 However, when $D_1\neq D_2$, the transformation idea fails to work. Luckily, we are able to find a Lyapunov function $E(u,v,w)$ defined by \eqref{La-3} for the system \eqref{1-1} under the condition \eqref{LT-31}, which satisfies
 %We need to estimate the key  term $\int_\Omega \nabla u\cdot \nabla v$ in other ways. To this aim, we  use the second and third equations of system \eqref{1-1} to express the term $\int_\Omega \nabla u\cdot \nabla v$  as follows (see \eqref{La-9} for details):
% \begin{equation}\label{ST-4}
% \begin{split}
 %\gamma \int_\Omega\nabla u\cdot\nabla v
% =&\frac{d}{dt}\int_\Omega \nabla w\cdot \nabla v+\left(D_1+D_2\right)\int_\Omega \Delta w\Delta v\\
 %&+\left(\beta+\delta\right)\int_\Omega \nabla w\cdot \nabla v-\alpha \int_\Omega \nabla u\cdot\nabla w.
%\end{split}
 %\end{equation}
 %On the other hand, from the second equation of system \eqref{1-1}, we have
 %\begin{equation}\label{ST-5}
 %\frac{1}{2}\frac{d}{dt}\int_\Omega |\nabla v|^2+\beta\int_\Omega  |\nabla v|^2+D_1\int_\Omega |\Delta v|^2 =\alpha \int_\Omega\nabla u\cdot \nabla v.
 %\end{equation}
 %Then we  multiply the \eqref{ST-1}, \eqref{ST-3}, \eqref{ST-4} and \eqref{ST-5} by appropriate constants respectively, and add them to  cancel the term $\int_\Omega \nabla u\cdot \nabla v$ such that
 \begin{equation}\label{ST-6}
 \frac{d}{dt} E(u,v,w)+ F(u,v,w)=0,
 \end{equation}
 where  $F(u,v,w)$ is defined by \eqref{La-4}. We remark that the form of $E(u,v,w)$ is quite different from the one (\ref{LY}) for $D_1=D_2$. To prove $E(u,v,w)$ can form a Laypunov function, we organize the estimates into a quadratic form which is the new idea developed in the paper.
 %To study the boundedness and large time behavior of solution for system \eqref{1-1} with $\tau_1=\tau_2=1$, we construct a Lyapunov function which was not found in previous papers.
 Then using \eqref{ST-6}, we show that under the condition \eqref{LT-31}, there exists  a constant $c_3>0$ such that $\|u\ln u\|_{L^1}+\|\nabla v\|_{L^2}+\|\nabla w\|_{L^2}\leq c_3$ and $\int_0^t\int_\Omega \frac{|\nabla u|^2}{u}\leq c_3$ (see Lemma \ref{Lc1} for details). The former  estimate leads to the boundedness of solutions in two dimensions and the later estimate gives the convergence properties of $u$.  The convergence of $v$ and $w$ can be derived by the parabolic comparison principle.  To study the decay rate, we first show that there exists a constant $\mu>0$ such that
  \begin{equation*}
  E(u,v,w)\leq \mu F(u,v,w),
  \end{equation*}
  which together with \eqref{ST-6} gives $E(u,v,w)\leq E(u_0,v_0,w_0) e^{-\frac{1}{\mu} t}$. Using the definition of $E(u,v,w)$ and noting the fact $\|u-\bar{u}\|_{L^1}\leq 2\bar{u}\int_\Omega u\ln \frac{u}{\bar{u}}$  in Lemma \ref{II},  the exponential decay of $\|u-\bar{u}\|_{L^1}$ under the condition \eqref{LT-31} is obtained. Then using the ideas in \cite{TW-M3AS-2013} or \cite{LM-NAR-2016}, we can show the decay rate of $\|u-\bar{u}\|_{L^\infty}$ and hence the exponential decay rate of $\|v-\frac{\alpha}{\beta}\bar{u}_0\|_{L^\infty}$ and $\|w-\frac{\gamma}{\delta}\bar{u}_0\|_{L^\infty}$.\\

In the end of this section, we remark that Theorem \ref{LT-3} only present some first-hand results  on the full ARKS model for $D_1\ne D_2$ under the parameter regime given in (\ref{LT-31}) and leave out many interesting questions due to technical difficulty. For example, whether the condition (\ref{LT-31}) is necessary for global existence of solutions and how solutions behave (in particular whether the solutions blow up) if the condition (\ref{LT-31}) fails remain unsolved in our paper. We hope our studies in this paper will provide useful clues to explore the ARKS model in future in the full parameter regime.

%For  the case $\tau_1=\tau_2=0$ and $\tau_1=1,\tau_2=0$, we can use the similar arguments to study the boundedness and large time behavior of solution. However, we should point out that there exist some differences between the case $\tau_2=1$ and $\tau_2=0$. In general case, we need more conditions to study the large time behavior than to show the existence of  global boundedness. As shown in Theomre \ref{BS}, the boundedness of solution exists under the condition $\xi\gamma D_1\geq \chi\alpha D_2$, however we need the condition   $\xi\gamma\geq \chi\alpha \max\Big\{\frac{D_2}{D_1},\frac{\delta}{\beta}\Big\}$ to show the global stability of solution.  Since  when $\tau_1=\tau_2=1$,  the key estimates for studying the boundedness and large time behavior of solution are obtained simultaneously by using  the Lyapunov function  (See Lemma \ref{Lc1} for details), that is why we need more conditions to obtain the boundedness of solution than the case $\tau_2=0$. We conjecture that when $\tau_1=\tau_2=1$, the global classical solution with uniform-in-time bound  will exist  under the same condition as  the case  $\tau_2=0$. However, due to the technical  difficulties, this problem is left as an open problem.
  \section{Some basic inequalities}  In what follows, without confusion,  we shall abbreviate $\int_\Omega fdx$ as $\int_\Omega f$ for simplicity. Moreover, we shall use $c_i (i=1,2,3, \cdots$) to denote a generic constant which may vary in the context.
For reader's convenience, we present some known inequalities for later use.
\begin{lemma}\label{II}
 Suppose that $f(x,t)$ is a positive function on $(x,t)\in\Omega \times (0,\infty)$.  Defined $\bar{f}=\frac{1}{|\Omega|}\int_\Omega f$, then it has that
\begin{equation}\label{II-1}
0\leq \frac{1}{2\bar{f}}\|f-\bar{f}\|_{L^1}^2\leq \int_\Omega f\ln \frac{f}{\bar{f}}\leq \frac{1}{\bar{f}}\|f-\bar{f}\|_{L^2}^2.
\end{equation}
\end{lemma}
\begin{proof}Using the Csisz\'{a}r-Kullback-Pinsker inequality (see \cite{CJMT} ), one has
\begin{equation}\label{II-2}
\int_\Omega f\ln \frac{f}{\bar{f}}\geq \frac{1}{2\bar{f}}\|f-\bar{f}\|_{L^1}^2.
\end{equation}
On the other hand, choosing $\psi=\frac{f}{\overline{f}}$ and using the fact that   $\psi\ln \psi-\psi+1\leq (\psi-1)^2$ for $\psi\geq 0$, it holds that
\begin{equation}\label{II-3}
\begin{split}
\int_\Omega f\ln \frac{f}{\bar{f}}
&\leq \bar{f}\int_\Omega\left[\frac{f}{\bar{f}}-1
+\left(\frac{f}{\bar{f}}-1\right)^2\right]=\frac{1}{\bar{f}}\|f-\bar{f}\|_{L^2}^2.
\end{split}
\end{equation}
Then the combination of \eqref{II-2} and $\eqref{II-3}$ gives \eqref{II-1}.
\end{proof}

\begin{lemma}\label{Key-e}
Let $\Omega$ be a bounded domain in $\R^2$ with smooth boundary. Then, for any $\varphi\in W^{3,2}(\Omega)$ satisfying $\frac{\partial \varphi}{\partial \nu}|_{\partial\Omega}=0$, there exists a positive constant $C$ depending only on $\Omega$ such that
\begin{equation}\label{key-e1}
\|\Delta \varphi\|_{L^3}\leq C(\|\nabla \Delta \varphi\|_{L^2}^\frac{2}{3}\|\nabla \varphi\|_{L^2}^\frac{1}{3}+\|\nabla \varphi\|_{L^2}).
\end{equation}
\end{lemma}
\begin{proof}
Using Gagliardo-Nirenberg inequality, we have
 \begin{equation}\label{Lvw-1}
 \begin{split}
 \|\Delta \varphi\|_{L^3}=\|\nabla \cdot \nabla \varphi\|_{L^3}
 &\leq \|D\nabla \varphi\|_{L^3}\leq c_1\|D^2 \nabla \varphi\|_{L^2}^\frac{2}{3}\|\nabla \varphi\|_{L^2}^\frac{1}{3}+c_2\|\nabla \varphi\|_{L^2},
 \end{split}
 \end{equation}
 where $|D^k \nabla \varphi|=(\sum_{|i|=k} |D^i \nabla \varphi|^2)^\frac{1}{2}$ and $i$ is a multi-index of order.
 On the other hand,  one can  check that
 \begin{equation}\label{Lvw-2}
 \|D^2 \nabla \varphi\|_{L^2}\leq c_3\|\nabla \varphi\|_{H^2}.
 \end{equation}
 Moreover, under the homogeneous Neumann boundary condition (i.e., $\frac{\partial \varphi }{\partial \nu}|_{\partial \Omega}=0$), it follows from \cite[Lemma 1]{Bou-JFA} that
$
 \|\nabla \varphi\|_{H^2}\leq c_4\|\Delta \varphi\|_{H^1},
$
 which applied to \eqref{Lvw-2} gives
 \begin{equation}\label{Lvw-3}
  \|D^2 \nabla \varphi\|_{L^2}\leq c_3c_4\|\Delta \varphi\|_{H^1}.
 \end{equation}
Note that $|\Delta \varphi|^2=\nabla \cdot (\nabla \varphi \Delta \varphi)-\nabla \varphi\cdot \nabla \Delta \varphi$. Then using the boundary condition $\frac{\partial \varphi }{\partial \nu}|_{\partial \Omega}=0$ and H\"{o}lder inequality, we have
 \begin{equation}\label{Lvw-4}
 \|\Delta \varphi\|_{L^2}^2=-\int_\Omega \nabla \varphi\cdot \nabla \Delta \varphi\\
\leq  \|\nabla \varphi\|_{L^2}\|\nabla \Delta \varphi\|_{L^2}.
 \end{equation}
Then substituting  \eqref{Lvw-3} into \eqref{Lvw-1}, and using \eqref{Lvw-4}, one derives
 \begin{equation*}
 \begin{split}
 \|\Delta \varphi\|_{L^3}\leq
 &c_5\left(\|\Delta \varphi\|_{H^1}^\frac{2}{3}\|\nabla \varphi\|_{L^2}^\frac{1}{3}+\|\nabla \varphi\|_{L^2}\right)\\
 = &c_5\left(\|\nabla\Delta \varphi\|_{L^2}+\|\Delta \varphi\|_{L^2}\right)^\frac{2}{3}\|\nabla \varphi\|_{L^2}^\frac{1}{3}+c_5\|\nabla \varphi\|_{L^2}\\
 \leq & c_6\left( \|\nabla \Delta \varphi\|_{L^2}+\|\nabla\Delta \varphi\|_{L^2}^\frac{1}{2}\|\nabla \varphi\|_{L^2}^\frac{1}{2}\right)^\frac{2}{3}\|\nabla \varphi\|_{L^2}^\frac{1}{3}+c_6\|\nabla \varphi\|_{L^2}\\
 \leq & c_6\left( 2\|\nabla \Delta \varphi\|_{L^2}+\|\nabla \varphi\|_{L^2}\right)^\frac{2}{3}\|\nabla \varphi\|_{L^2}^\frac{1}{3}+c_6\|\nabla \varphi\|_{L^2}\\
 \leq &c_7 \|\nabla \Delta \varphi\|_{L^2}^\frac{2}{3}\|\nabla \varphi\|_{L^2}^\frac{1}{3}+c_7\|\nabla \varphi\|_{L^2}\\
 \end{split}
 \end{equation*}
 which yields \eqref{key-e1}, and hence completes the proof.
\end{proof}

%The local existence theorem of (\ref{3-1}) can be proved by the fixed point theorem and maximum principle along the same line shown in \cite{Tao-W}.
%\begin{lemma}\label{LS}
%Assume that $0\leq(u_0,\tau_1v_0)\in [W^{1,\infty}(\Omega)]^2$. Then there exist $T_{max}\in(0,\infty]$ and a unique triple $(u,v,w)$ of nonnegative functions  from $C(\bar{\Omega}\times[0,T_{max}))\cap C^{2,1}(\bar{\Omega}\times(0,T_{max}))$ solving $\eqref{3-1}$ classically in $\Omega\times(0,T_{max})$. Moreover $u>0$ in $\Omega\times(0,T_{max})$ and
%\begin{equation}\label{BC}
%if ~T_{max}<\infty, ~~then~~\norm{u(\cdot,t)}_{L^\infty}\to\infty ~~ as~~~t\nearrow T_{max}.
%\end{equation}
%\end{lemma}
\section{Boundedness criterion and Lyapunov function}
%In this section,  we first establish  the boundedness criterion for the solution of system \eqref{1-1} inspired by \cite[lemma 3.2]{BBTW-2015}, which implies the boundedness of  $\|u\|_{L^\infty}$ can be reduce to show the boundedness of $\|u\|_{L^p}$ with $p>\max\{1,\frac{n}{2}\}$;  Furthermore, we will construct a new Lyapunov function for system \eqref{1-1} with $\tau_1>0$ (i.e, $\tau_1=\tau_2=1$ or  $\tau_1=1,\tau_2=0$).
 \subsection{Local existence}
The local existence theorem of  system (\ref{1-1}) can be proved by the fixed point theorem and maximum principle along the same line as in \cite{TW-M3AS-2013}.
Hence we only present the results without proof for brevity.
\begin{lemma}\label{LS}
Let $\Omega $ be a bounded domain in $\R^n(n\geq 2)$ with smooth boundary. Suppose that $0\leq (u_0,v_0,w_0)\in [W^{1,\infty}(\Omega)]^3$.   Then there exist a $T_{max}\in(0,\infty]$ such that the system \eqref{1-1}  has a unique  solution $(u,v,w)$ of nonnegative functions  from $[C^0(\bar{\Omega}\times[0,T_{max}))\cap C^{2,1}(\bar{\Omega}\times(0,T_{max}))]^3$. Moreover $u>0$ in $\Omega\times(0,T_{max})$ and
\begin{equation}\label{BC}
if ~T_{max}<\infty, ~~then~~\norm{u(\cdot,t)}_{L^\infty}\to\infty ~~ as~~~t\nearrow T_{max}.
\end{equation}
Furthermore, the cell mass is conservative:
\begin{equation}\label{L1-u}
\norm{u(\cdot,t)}_{L^1}=\norm{u_0}_{L^1}.
\end{equation}
\end{lemma}
\subsection{Boundedness criterion}
To extend the local solutions to global ones, we derive a boundedness criterion for the solution of system \eqref{1-1}. The idea of our proof is essentially inspired by \cite[lemma 3.2]{BBTW-2015} and we present necessary details  below for clarity.
\begin{lemma} \label{BC}
Suppose the conditions in Lemma \ref{LS} hold.
%Let $n\geq 1$ and $\Omega \subset\R^n$ be a bounded domain with smooth boundary, and suppose  the hypotheses (H1)-(H3) hold. Let $(u_0,v_0)$ be  the initial data as in Theorem \ref{GB} and
Let  $(u,v,w)$ be the solution of system $\eqref{1-1}$ defined on its maximal existence time interval $[0,T_{max})$. If there exist $p>\frac{n}{2}$ and a constant $M_0$ such that
\begin{equation*}\label{BCC}
\sup\limits_{t\in(0,T_{max})}\|u(\cdot,t)\|_{L^p}\leq M_0,
\end{equation*}
then one can find a constant $C>0$ independent of $t$ such that
\begin{equation}\label{BCR}
\|u(\cdot,t)\|_{L^\infty}+\|v(\cdot,t)\|_{W^{1,\infty}}+\|w(\cdot,t)\|_{W^{1,\infty}}\leq C\ \ \mathrm{for\ all}\ \  t\in (0,T_{max}).
\end{equation}
Furthermore, there exists $\sigma\in(0,1)$ such that for all $t>1$
\begin{equation}\label{A10*-1}
\|u\|_{C^{\sigma,\frac{\sigma}{2}}(\bar{\Omega}\times[t,t+1])}\leq C  .
\end{equation}

\end{lemma}
\begin{proof}
%We give only the proof for the case $\tau_1=\tau_2=1$ due to  the proof for the other cases can be proceeded similarly.
Since $\|u(\cdot,t)\|_{L^p}\leq M_0$, then applying  {\color{red}the  parabolic  regularity estimates  in \cite[Lemma 1]{Kowalczyk-S-JMAA-2008}  } to the second  and third equations of system \eqref{1-1}  we have
\begin{equation}\label{BC-2}
\|\nabla v(\cdot,t)\|_{L^r}+\|\nabla w(\cdot,t)\|_{L^r}\leq c_1, \ \mathrm{for\ all} \ t\in(0,T_{max})
\end{equation}
where
\begin{equation}\label{BC-3}
r\in\begin{cases}
[1,\frac{np}{n-p}), \ \ &\mathrm{if}\ \ \  p\leq n,\\
[1,\infty],\ \ &\mathrm{if}\ \ \  p>n.\\
\end{cases}
\end{equation}
 Without loss of generality, we assume that $\frac{n}{2}<p\leq n$ which yields $\frac{np}{n-p}>n$. Then we can find a constant $r>0$ with  $n<r<\frac{np}{n-p}$ such that $\eqref{BC-2}$ holds.
%\begin{equation}\label{BC-6}
%q>\frac{r}{r+1}
%\end{equation}
%and
%\begin{equation}\label{BC-7}
%n<q<r.
%\end{equation}
Now, for each $T\in(0,T_{max})$,  we define
\begin{equation}\label{DM}
M(T):=\sup\limits_{t\in(0,T)}\|u(\cdot,t)\|_{L^\infty},
\end{equation}
which is finite due to the local existence results in Lemma \ref{LS}. Next, we will estimate $M(T)$. Fix $t\in(0,T)$ and let $t_0=(t-1)_+$. Then applying  the variation-of-constants formula to the first equation of system \eqref{1-1}, we get
\begin{equation*}\label{BC-8}
\begin{split}
u(\cdot,t)%=&e^{(t-t_0)\Delta} u(\cdot,t_0)-\chi \int_{t_0}^t e^{(t-s)\Delta}\nabla \cdot(u(\cdot,s)\nabla v(\cdot,s))ds\\
%&+\gamma \int_{t_0}^t e^{(t-s)\Delta }u(\cdot,s)F(v(\cdot,s))ds-\int_{t_0}^t e^{(t-s)\Delta} u(\cdot,s)h(u(\cdot,s))ds\\
= &e^{(t-t_0)\Delta} u(\cdot,t_0)-\chi \int_{t_0}^t e^{(t-\tau)\Delta}\nabla \cdot(u(\cdot,\tau)\nabla v(\cdot,\tau))d\tau\\
&+\xi \int_{t_0}^t e^{(t-\tau)\Delta}\nabla \cdot(u(\cdot,\tau)\nabla w(\cdot,\tau))d\tau\end{split}
\end{equation*}
which implies
\begin{equation}\label{BC-9}
\begin{split}
\|u(\cdot,t)\|_{L^\infty}
\leq& \|e^{(t-t_0)\Delta}u(\cdot,t_0)\|_{L^\infty}+\chi\int_{t_0}^t\|e^{(t-\tau)\Delta}\nabla \cdot(u(\cdot,\tau)\nabla v(\cdot,\tau))\|_{L^\infty}d\tau\\
       &+\xi\int_{t_0}^t\|e^{(t-\tau)\Delta}\nabla \cdot(u(\cdot,\tau)\nabla w(\cdot,\tau))\|_{L^\infty}d\tau\\
 = & I_1+I_2+I_3.
\end{split}
\end{equation}
%If $t\leq 1$, then the application of the order preserving properties of the Neumann heat semigroup gives
%The argument in \cite[Lemma 3.2]{BBTW-2015} has shown that there is a constant $c_2>0$ such that
%\begin{equation}\label{I1}
%I_1\leq \|u_0\|_{L^\infty}\leq c_2.
%\end{equation}
We first estimate the term $I_1$. If $t\leq 1$, then $t_0=0$ and we can use the maximum principle for the heat equation to obtain
\begin{equation}\label{I10}
I_1=\|e^{t\Delta} u_0\|_{L^\infty}\leq \|u_0\|_{L^\infty}\leq c_2,
\end{equation}
whereas in the case $t>1$ and $t_0=t-1$, we use the standard $L^p$-$L^q$ estimates for  $(e^{\tau\Delta})_{\tau\geq 0}$ to derive
\begin{equation}\label{I11}
I_1\leq c_3 \|u(\cdot,t_0)\|_{L^p}\leq c_3 M_0=c_4.
\end{equation}
Moreover,  since $r>n$, we can fix a number $q>n$ satisfying $q\in (\frac{r}{r+1},r)$. Then by the H\"{o}lder inequality, interpolation inequality and (\ref{DM}), we  can find $\zeta=\frac{r(q-1)+q}{rq}\in(0,1)$ such that
\begin{equation*}\label{BC-14}
\begin{split}
\|u(\cdot, \tau)\nabla v(\cdot, \tau)\|_{L^q}&\leq \|u(\cdot, \tau)\|_{L^{\frac{rq}{r-q}}}\|\nabla v(\cdot, \tau)\|_{L^r}\\
&\leq \|u(\cdot, \tau)\|_{L^\infty}^{1-\frac{r-q}{rq}}\|u(\cdot, \tau)\|_{L^1}^{\frac{r-q}{rq}}\|\nabla v(\cdot, \tau)\|_{L^r}\\
&\leq c_5 M^\zeta(T).
\end{split}
\end{equation*}
Similarly,  we have
\begin{equation*}\label{BC-14}
\begin{split}
\|u(\cdot, \tau)\nabla w(\cdot, \tau)\|_{L^q}
%&\leq \|u(\cdot, s)\|_{L^{\frac{rq}{r-q}}}\|\nabla w(\cdot, s)\|_{L^r}\\
%&\leq \|u(\cdot, s)\|_{L^\infty}^{1-\frac{r-q}{rq}}\|u(\cdot, s)\|_{L^1}^{\frac{r-q}{rq}}\|\nabla w(\cdot, s)\|_{L^r}\\
&\leq c_6 M^\zeta(T).
\end{split}
\end{equation*}
Since $t-t_0\leq 1$, we have $\int_{t_0}^t(t-s)^{-\frac{1}{2}-\frac{n}{2q}}ds=\int_0^{t-t_0}\sigma^{-\frac{1}{2}-\frac{n}{2q}}d\sigma\leq \int_0^1\sigma^{-\frac{1}{2}-\frac{n}{2q}}d\sigma={\frac{2q}{q-n}}$ thanks to $q>n$. Then by the smoothing properties of $(e^{\tau\Delta})_{\tau\geq 0}$ (see  \cite[Lemma 1.3]{W-JDE-2010}), we derive
\begin{equation}\label{BC-13}
\begin{split}
I_2+I_3
&\leq c_7 \int_{t_0}^t (t-\tau)^{-\frac{1}{2}-\frac{n}{2q} }(\|u(\cdot,\tau)\nabla v(\cdot,\tau)\|_{L^q}+\|u(\cdot,\tau)\nabla w(\cdot,\tau)\|_{L^q})d\tau\\
&\leq c_8 M^\zeta(T)\int_{t_0}^t (t-\tau)^{-\frac{1}{2}-\frac{n}{2q}}d\tau\\
&\leq \frac{2qc_8}{q-n}M^\zeta(T):=c_9M^\zeta(T).
\end{split}
\end{equation}
 Substituting  \eqref{I10}, \eqref{I11} and \eqref{BC-13} into \eqref{BC-9}, we can find a constant $c_{10}>0$ such that
\begin{equation*}
\|u(\cdot,t)\|_{L^\infty}\leq c_9M^\zeta(T)+c_{10},\ \ \mathrm{for\ all}\ \ t\in(0,T),
\end{equation*}
which implies
\begin{equation}\label{BC-15}
M(T)\leq c_9M^\zeta(T)+ c_{10}, \ \ \mathrm{for\ all}\ \ T\in(0,T_{max}).
\end{equation}
Since $0<\zeta<1$, from \eqref{BC-15} one has
\begin{equation*}
M(T)\leq \max\Big\{\left(\frac{c_{10}}{c_9}\right)^\frac{1}{\zeta},(2c_9)^\frac{1}{1-\zeta}\Big\},\ \ \mathrm{for\ all}\ \ T\in(0,T_{max}),
\end{equation*}
which implies $\|u(\cdot,t)\|_{L^\infty}\leq c_{11}$ for all $t\in(0,T_{max})$.   Furthermore the combination of  $\eqref{BC-2}$ and \eqref{BC-3} gives (\ref{BCR}).

At last, from (\ref{BCR}) we know that $\chi u\nabla v$ and $\xi u\nabla w$  are bounded in $L^\infty(\Omega\times(0,\infty))$. Then applying the standard parabolic regularity theory  (e.g. see \cite[Theorem 1.3]{RT} and \cite[Lemma 3.2]{Tao-Winkler-SIMA-2015}) and parabolic Schauder theory \cite{La}, we immediately obtain the estimate \eqref{A10*-1}. Then the proof of  Lemma \ref{BC} is completed.
\end{proof}

\subsection{Lyapunov function} As mentioned in Remark \ref{rem1}, we consider the case $D_1\ne D_2$ or $\beta\ne \delta$ which implies that $\theta_1>0$ from \eqref{LT-31}. When $D_1=D_2$, the boundedness of solutions shown in  Theorem \ref{LT-3} has been proved in \cite{TW-M3AS-2013} with $\beta=\delta$ and in \cite{LT-M2AS-2015, J-JMAA-2015} with $\beta\neq \delta$  by constructing entropy inequality based on an idea of using the transformation $s=\xi w-\chi v$. However this transformation is no longer helpful for the case $D_1\neq D_2$.  Hence, we need to find a new way.  Here we achieve our results by constructing a Lyapunov function for  the system \eqref{1-1}.
%Recall that $\theta_1>0$ if $D_1 \ne D_2$ or $\beta\ne \delta$ which is our focus of this paper.
First, we define
 \begin{equation}\label{La-3}
 \begin{aligned}
E(u,v,w):=&\frac{\theta_1}{2\xi\chi}\int_\Omega u\ln\frac{u}{\bar{u}}
+\frac{\theta_2}{4\xi\alpha}\int_\Omega|\nabla v|^2+\frac{\theta_2}{4\gamma\chi}\int_\Omega |\nabla w|^2-\int_\Omega \nabla w\cdot\nabla v
\end{aligned}
\end{equation}
and
\begin{equation}\label{La-4}
\begin{split}
F(u,v,w):=&\frac{\theta_1}{2\xi\chi}\int_\Omega \frac{|\nabla u|^2}{u}+\frac{\theta_2D_1}{2\xi\alpha}\int_\Omega |\Delta v|^2+\frac{\theta_2D_2}{2\gamma\chi}\int_\Omega |\Delta w|^2+ \frac{\theta_2\beta}{2\xi\alpha}\int_\Omega |\nabla v|^2\\
&+\frac{\theta_2\delta}{2\gamma\chi}\int_\Omega |\nabla w|^2
-\left(D_1+D_2\right)\int_\Omega \Delta w\Delta v-\left(\beta+\delta\right)\int_\Omega \nabla w\cdot \nabla v,
\end{split}
\end{equation}
{\color{black}where $\theta_1:=\xi\gamma-\chi\alpha$ and $\theta_2:=\xi\gamma+\chi\alpha$.}
Then, we will show that $E(u,v,w)$ is indeed a Lyapunov function under \eqref{LT-31}. More precisely, we have the following results.
\begin{lemma} \label{La}  Let $(u,v,w)$ be the solution of system \eqref{1-1}.  Then we have
\begin{equation}\label{La-2}
\frac{d}{dt} E(u,v,w)+F(u,v,w)=0
\end{equation}
where $E(u,v,w)$ and $F(u,v,w)$ are defined by \eqref{La-3} and \eqref{La-4}, respectively. Moreover, if \eqref{LT-31} holds, then
\begin{equation}\label{Lc1-11}
E(u,v,w)\geq 0\ \ \mathrm{and}\ \  F(u,v,w)\geq 0\ \ \mathrm{for\ all}\ \  t> 0.
\end{equation}
\end{lemma}
\begin{proof}
Multiplying the first equation of system \eqref{1-1} by $\ln \frac{ u}{\bar {u}} $, we have
\begin{equation}\label{La-5}
\frac{d}{dt}\int_\Omega u\ln \frac{u}{\bar{u}}+\int_\Omega \frac{|\nabla u|^2}{u}=\chi\int_\Omega \nabla u\cdot \nabla v-\xi \int_\Omega \nabla u\cdot \nabla w.
\end{equation}
Similarly, we  multiply the second and  third equations of system \eqref{1-1}  by $-\Delta v$ and $-\Delta w$, respectively, to obtain
\begin{equation}\label{La-6}
\frac{1}{2}\frac{d}{dt}\int_\Omega |\nabla v|^2+D_1\int_\Omega |\Delta v|^2+\beta \int_\Omega |\nabla v|^2=\alpha\int_\Omega \nabla u\cdot \nabla v
\end{equation}
and
\begin{equation}\label{La-7}
\frac{1}{2}\frac{d}{dt}\int_\Omega |\nabla w|^2+D_2\int_\Omega |\Delta w|^2+\delta \int_\Omega |\nabla w|^2=\gamma\int_\Omega \nabla u\cdot \nabla w.
\end{equation}
 Multiplying \eqref{La-5} by $\frac{\theta_1}{2\xi\chi}$, \eqref{La-6} by $\frac{\theta_2}{2\xi\alpha}$ and \eqref{La-7} by $\frac{\theta_2}{2\gamma\chi}$, and adding them, we end up with
\begin{equation}\label{La-8}
\begin{split}
&\frac{d}{dt}\left(\frac{\theta_1}{2\xi\chi}\int_\Omega u\ln\frac{u}{\bar{u}}+\frac{\theta_2}{4\xi\alpha}\int_\Omega|\nabla v|^2
+\frac{\theta_2}{4\gamma\chi}\int_\Omega |\nabla w|^2\right) +\frac{\theta_1}{2\xi\chi}\int_\Omega \frac{|\nabla u|^2}{u}\\
&+\frac{\theta_2D_1}{2\xi\alpha}\int_\Omega |\Delta v|^2
+\frac{\theta_2 D_2}{2\gamma\chi}\int_\Omega |\Delta w|^2
+ \frac{\theta_2\beta}{2\xi\alpha}\int_\Omega |\nabla v|^2+\frac{\theta_2\delta}{2\gamma\chi}\int_\Omega |\nabla w|^2\\
&=\gamma\int_\Omega \nabla u\cdot \nabla v+\alpha\int_\Omega \nabla u\cdot \nabla w.
\end{split}
\end{equation}
On the other hand, the second and  third equations of system \eqref{1-1} give us that
\begin{equation*}
\begin{split}
\gamma \int_\Omega \nabla u\cdot \nabla v
&=\int_\Omega \nabla (w_t+\delta w-D_2\Delta w)\cdot \nabla v\\
&=\int_\Omega \nabla w_t \cdot \nabla v+\delta \int_\Omega \nabla w\cdot\nabla v-D_2\int_\Omega \nabla (\Delta w)\cdot \nabla v\\
&=\frac{d}{dt}\int_\Omega \nabla w\cdot \nabla v+\int_\Omega \Delta w v_t+\delta \int_\Omega \nabla w\cdot\nabla v+D_2\int_\Omega \Delta w \Delta v\\
&=\frac{d}{dt}\int_\Omega \nabla w\cdot \nabla v+ \int_\Omega \Delta w (D_1 \Delta v+\alpha u-\beta v)\\
&\ \ \ \ +\delta \int_\Omega \nabla w\cdot\nabla v+D_2\int_\Omega \Delta w \Delta v\\
 &=\frac{d}{dt}\int_\Omega \nabla w\cdot \nabla v+\left(D_1+D_2\right)\int_\Omega \Delta w\Delta v \\
 &\ \ \ \ +\left(\beta+\delta\right)\int_\Omega \nabla w\cdot \nabla v-\alpha\int_\Omega \nabla u\cdot \nabla w
\end{split}
\end{equation*}
%Moreover, using the second equation and third equation of system \eqref{1-1} with $\tau_1=\tau_2=1$, we can  obtain
%\begin{equation*}
%\begin{split}
%-\frac{d}{dt}\int_\Omega \nabla v\cdot \nabla w
%&=-\int_\Omega \nabla v_t\cdot \nabla w-\int_\Omega \nabla  v\cdot \nabla w_t\\
%&=\int_\Omega v_t \Delta w+\int_\Omega w_t \Delta v\\
%&=\int_\Omega (D_1\Delta v+\alpha u-\beta v)\Delta w+\int_\Omega \Delta v(D_2\Delta w+\gamma u-\delta w)\\
%&=(D_1+D_2)\int_\Omega \Delta v\Delta w+(\beta+\delta)\int_\Omega \nabla v\cdot \nabla w\\
%&\ \ \ \ -\alpha\int_\Omega \nabla u\cdot \nabla w-\gamma \int_\Omega \nabla u\cdot \nabla v,
%\end{split}
%\end{equation*}
which yields
\begin{equation}\label{La-9}
\begin{split}
\gamma \int_\Omega \nabla u\cdot \nabla v+\alpha \int_\Omega  \nabla u\cdot \nabla w
&=\frac{d}{dt}\int_\Omega \nabla w\cdot \nabla v+\left(D_1+D_2\right)\int_\Omega \Delta w\Delta v\\
&\ \ \ \ +\left(\beta+\delta\right)\int_\Omega \nabla w\cdot \nabla v.
\end{split}
\end{equation}
The combination of \eqref{La-8} and \eqref{La-9} gives \eqref{La-2}.

Next, we will show the nonnegative of $E(u,v,w)$ and $F(u,v,w)$ under \eqref{LT-31}. First, we rewrite  $E(u,v,w)$ in \eqref{La-3} as
\begin{equation*}\label{Lc-1}
\begin{split}
E(u,v,w)=\frac{\theta_1}{2\xi\chi}\int_\Omega u\ln\frac{u}{\bar{u}}+\int_\Omega\Theta_1^T A_1 \Theta_1
\end{split}
\end{equation*}
where $\Theta_1^T$ denotes the transpose of $\Theta_1$ and
$$\ \Theta_1=\begin{bmatrix}\nabla v\\[1mm]\nabla w \end{bmatrix}\ \  \mathrm{and} \ \ \ A_1=\begin{bmatrix}\frac{\theta_2}{4\xi\alpha}& -\frac{1}{2} \\[1mm]-\frac{1}{2} &\frac{\theta_2}{4\gamma\chi}\end{bmatrix}.$$
Since $\theta_1>0$, one has  $\theta_2^2>\theta_2^2-\theta_1^2=4\xi\gamma\chi\alpha$. This implies the matrix $A_1$ is positive definite and  hence there exists a constant $c_1>0$ such that
\begin{equation}\label{PE*}
E(u,v,w)\geq \frac{\theta_1}{2\xi\chi}\int_\Omega u\ln\frac{u}{\bar{u}}+c_1\int_\Omega (|\nabla v|^2+|\nabla w|^2)\geq 0,
\end{equation}
{\color{black}where we have used the fact $\int_\Omega u\ln\frac{u}{\bar{u}}\geq 0$ from Lemma \ref{II}.}
Similarly, we  rewrite $F(u,v,w)$ as
\begin{equation}\label{Lc-2}
\begin{split}
 F(u,v,w)=&\frac{\theta_1}{2\xi\chi}\int_\Omega \frac{|\nabla u|^2}{u}+\int_\Omega\Theta_2^T A_2 \Theta_2+\int_\Omega \Theta_1^T A_3\Theta_1,
\end{split}
\end{equation}
where
\begin{equation}\label{A2-A3}
\ \ \ \  \ \Theta_2=\begin{bmatrix}\Delta v\\[1mm]\Delta w \end{bmatrix}, \ \
A_2=\begin{bmatrix}\frac{\theta_2D_1}{2\xi\alpha}& -\frac{D_1+D_2}{2} \\[1mm]-\frac{D_1+D_2}{2} &
\frac{\theta_2 D_2}{2\gamma\chi}\end{bmatrix} \ \mathrm{and} \ \ A_3=\begin{bmatrix}\frac{\theta_2\beta}{2\xi\alpha}& -\frac{\beta+\delta}{2} \\[1mm]-\frac{\beta+\delta}{2} &\frac{\theta_2\delta}{2\gamma\chi}\end{bmatrix}.
\end{equation}
Clearly,  the matrix  $A_2$ is nonnegative definite if
\begin{equation*}
\theta_2^2D_1D_2-\frac{(D_1+D_2)^2(\theta_2^2-\theta_1^2)}{4}\geq0.
\end{equation*}
Similarly,   the matrix $A_3$ is nonnegative definite under the condition
\begin{equation*}
\theta_2^2\beta\delta-\frac{(\beta+\delta)^2(\theta_2^2-\theta_1^2)}{4}\geq0.
\end{equation*}
Hence, the nonnegativity  of the matrices $A_2$ and $A_3$ are satisfied  simultaneously if
\begin{equation*}
\begin{cases}
4\theta_2^2D_1D_2-(D_1+D_2)^2(\theta_2^2-\theta_1^2)\geq0,\\
4\theta_2^2\beta\delta-(\beta+\delta)^2(\theta_2^2-\theta_1^2)\geq0,
\end{cases}
\end{equation*}
which is equivalent to
\begin{equation}\label{Lc-5}
\begin{cases}
\theta_2^2(D_1-D_2)^2\leq (D_1+D_2)^2\theta_1^2,\\
\theta_2^2(\beta-\delta)^2\leq (\beta+\delta)^2\theta_1^2.\\
\end{cases}
\end{equation}
{\color{black}One can check that  \eqref{Lc-5} holds  if  $\frac{\xi\gamma}{\chi\alpha}\geq\max\{\frac{D_1}{D_2},\frac{D_2}{D_1},\frac{\beta}{\delta},\frac{\delta}{\beta}\}$ or $\frac{\xi\gamma}{\chi\alpha}\leq \min\{\frac{D_1}{D_2},\frac{D_2}{D_1},\frac{\beta}{\delta},\frac{\delta}{\beta}\}$. However, the latter is impossible due to $\theta_1>0$.} Hence,  if  \eqref{LT-31} holds, one has $E(u,v,w)\geq 0$ and $F(u,v,w)\geq 0$.  The proof of  \eqref{Lc1-11} is completed.
\end{proof}

\section{Proof of Theorem \ref{LT-3}}
In this section, we are devoted to proving  Theorem \ref{LT-3} based on  the Lyapunov function constructed in Lemma \ref{La}.

%If $D_1=D_2$ and $\beta=\delta$, the boundedness and large time behavior of solutions stated in Theorem \ref{LT-3} have been studied in \cite[Proposition 2.6]{TW-M3AS-2013}. Below we only focus on the case $D_1\neq D_2$ or $\beta\neq \delta$.
   \subsection{Boundedness of solutions}
In this subsection, we  show the boundedness of solutions for system \eqref{1-1}  under the condition \eqref{LT-31}. First, we give  a core lemma concerning the boundedness and asymptotical behavior of solution for system \eqref{1-1} in two dimensions.
\begin{lemma} \label{Lc1}
%Let $D_1\neq D_2$ and $\tau_1=\tau_2=1$. Suppose $E(u,v,w)$ and $F(u,v,w)$ are the functions  defined by \eqref{La-3} and \eqref{La-4} respectively.
% Let $\Omega$ be a bounded domain  in $\R^2$ with smooth boundary. Suppose that $ (u_0,v_0,w_0)\in [W^{1,\infty}(\Omega)]^3$ and    \eqref{LT-31} hold.
Suppose that $ (u_0,v_0,w_0)\in [W^{1,\infty}(\Omega)]^3$ and    \eqref{LT-31} hold.  Then the solution $(u,v,w)$ of system \eqref{1-1}  satisfies
\begin{equation}\label{Lc1-12}
\|u\ln u\|_{L^1}+\|\nabla v\|_{L^2}+\|\nabla w\|_{L^2}\leq C
\end{equation}
and
\begin{equation}\label{Lc1-13}
\int_0^t\int_\Omega \frac{|\nabla u|^2}{u}\leq C,
\end{equation}
where $C>0$ is a constant independent of $t$.
\end{lemma}
\begin{proof}
The nonnegativity of  $E(u,v,w)$ and $F(u,v,w)$ has been proved in Lemma \ref{La} under the condition \eqref{LT-31}. Then integrating \eqref{La-2} and using \eqref{PE*} and \eqref{Lc-2}, along  with the nonnegativity  of  $A_2$ and $A_3$, we have two positive constants $c_1,c_2$ such that
 \begin{equation}\label{Lc1-14}
 \frac{\theta_1}{2\xi\chi}\int_\Omega u\ln\frac{u}{\bar{u}}+c_1\int_\Omega (|\nabla v|^2+|\nabla w|^2)\leq c_2,
 \end{equation}
 which, together with the fact $\int_\Omega u\ln\frac{u}{\bar{u}}\geq 0$ from Lemma \ref{II}, gives
 %Noting  the fact $-u\ln u\leq \frac{1}{e}$ for all $u\geq 0$ and  using \eqref{Lc1-14}, one can find a constant $c_3>0$ such that
  %\begin{equation*}
 %\begin{split}
%c_1\int_\Omega (|\nabla v|^2+|\nabla w|^2)
%&\leq c_2 -\frac{\theta_1}{2\xi\chi}\int_\Omega u\ln \frac{u}{\bar{u}}\\
%&=  c_2  -\frac{\theta_1}{2\xi\chi}\int_\Omega u\ln u+ \frac{\theta_1}{2\xi\chi}\ln \bar{u}\int_\Omega u\\
%&\leq  c_2+\frac{|\Omega|}{e}\cdot \frac{\theta_1}{2\xi\chi}+\frac{\theta_1}{2\xi\chi}|\Omega|\bar{u}\ln \bar{u}\leq c_3,
%\end{split}
 %\end{equation*}
% which implies
 \begin{equation}\label{Lc1-15*}
 \|\nabla v\|_{L^2}^2+\|\nabla w\|_{L^2}^2\leq\frac{c_2}{c_1}=c_3.
 \end{equation}
On the other hand, from \eqref{Lc1-14}, we directly obtain
   \begin{equation*}\label{Lc1-15}
  \frac{\theta_1}{2\xi\chi}\int_\Omega u\ln u \leq c_2+\frac{\theta_1}{2\xi\chi}|\Omega|\bar{u}\ln \bar{u}\leq c_4,
 \end{equation*}
 which, along with  the fact  $-u\ln u\leq \frac{1}{e}$ for all $u\geq 0$, gives
  \begin{equation}\label{Lc1-16}
 \int_\Omega |u\ln u|\leq \int_\Omega \Big| u\ln u+\frac{1}{e}-\frac{1}{e}\Big|\leq \int_\Omega \left(u\ln u+\frac{1}{e}\right)+\int_\Omega \frac{1}{e}\leq \frac{2\xi\chi c_4}{\theta_1} +\frac{2|\Omega|}{e}.
 \end{equation}
 Then the combination of  \eqref{Lc1-15*} and \eqref{Lc1-16} gives \eqref{Lc1-12}. Hence the proof of this lemma is completed.
 \end{proof}

 \begin{lemma}\label{L2e}
Let the assumptions in Lemma \ref{Lc1} hold. Then the solution $(u,v,w)$ of system \eqref{1-1}  satisfies
\begin{equation}\label{L2e-1}
\|u(\cdot,t)\|_{L^2}\leq C
\end{equation}
where the constant $C>0$ is independent of $t$.
\end{lemma}
\begin{proof}
Multiplying the first equation of system \eqref{1-1} by $u$ and integrating it by parts, we have
\begin{equation}\label{L2e-2}
\begin{split}
\frac{1}{2}\frac{d}{dt}\int_\Omega u^2+\int_\Omega |\nabla u|^2
&=\chi \int_\Omega u\nabla u\cdot \nabla v-\xi\int_\Omega u\nabla u\cdot \nabla w\\
&=-\frac{\chi}{2}\int_\Omega u^2\Delta v+\frac{\xi}{2}\int_\Omega u^2\Delta v\\
&\leq c_1\|u\|_{L^3}^2(\|\Delta v\|_{L^3}+\|\Delta w\|_{L^3}).\\
\end{split}
\end{equation}
Noting the fact $\|u\ln u \|_{L^1}\leq c_2$ and $\|u\|_{L^1}\leq c_3$,  one can find a small  $\varepsilon>0$ such that
\begin{equation}\label{L2e-3}
\|u\|_{L^3}^2 =\left(\|u\|_{L^3}^3\right)^\frac{2}{3}\leq \left(\varepsilon\|\nabla u\|_{L^2}^2+1\right)^\frac{2}{3}\leq \varepsilon \|\nabla u\|_{L^2}^\frac{4}{3}+c_4,
\end{equation}
where we have used the following fact (see \cite{Nagai-Funk}): when $n=2$,  for any $\varepsilon>0$, there exists a constant $C_\varepsilon$ such that
\begin{equation*}\label{L3 estimate}
\norm{u}_{L^3}\leq \varepsilon \norm{\nabla u}_{L^2}^\frac{2}{3}\norm{u\ln u}_{L^1}^\frac{1}{3}+C_\varepsilon(\|u\ln u\|_{L^1}+\|u\|_{L^1}^\frac{1}{3}).
\end{equation*}
On the other hand, noting the facts $\frac{\partial v}{\partial \nu}\Big|_{\partial\Omega}=\frac{\partial w}{\partial \nu}\Big|_{\partial\Omega}=0$ on $\partial \Omega$ and using the boundedness of $\|\nabla v\|_{L^2}$ and $\|\nabla w\|_{L^2}$ (see \eqref{Lc1-12}),   from Lemma \ref{Key-e},  one has
\begin{equation}\label{L2e-4}
\begin{split}
&\|\Delta v\|_{L^3}+\|\Delta w\|_{L^3}\\
&\leq c_5(\|\nabla \Delta v\|_{L^2}^\frac{2}{3}\|\nabla v\|_{L^2}^\frac{1}{3}+\|\nabla v \|_{L^2})+c_5(\|\nabla \Delta w\|_{L^2}^\frac{2}{3}\|\nabla w\|_{L^2}^\frac{1}{3}+\|\nabla w \|_{L^2})\\
&\leq c_6(\|\nabla \Delta v\|_{L^2}^\frac{2}{3}+\|\nabla \Delta w\|_{L^2}^\frac{2}{3}+1).
\end{split}
\end{equation}
%and
%\begin{equation}\label{L2e-5}
%\|\Delta w\|_{L^3}\leq c_6(\|\nabla \Delta w\|_{L^2}^\frac{2}{3}+1).
%\end{equation}
Then combining \eqref{L2e-3} and  \eqref{L2e-4}, and using Young's inequality and noting  the fact $\varepsilon>0$ is small, we  find  a small $\eta>0$ such that
\begin{equation}\label{L2e-6}
\begin{split}
&c_1\|u\|_{L^3}^2(\|\Delta v\|_{L^3}+\|\Delta w\|_{L^3})\\
&\leq c_7\left(\varepsilon\|\nabla u\|_{L^2}^\frac{4}{3}+c_4\right)\left(\|\nabla \Delta v\|_{L^2}^\frac{2}{3}+\|\nabla \Delta w\|_{L^2}^\frac{2}{3}+1\right)\\
&=c_7\varepsilon \|\nabla u\|_{L^2}^\frac{4}{3}\left(\|\nabla \Delta v\|_{L^2}^\frac{2}{3}+\|\nabla \Delta w\|_{L^2}^\frac{2}{3}\right)+c_7\varepsilon\|\nabla u\|_{L^2}^\frac{4}{3} \\
&\ \ \ \ +c_1c_7\left(\|\nabla \Delta v\|_{L^2}^\frac{2}{3}+\|\nabla \Delta w\|_{L^2}^\frac{2}{3}\right)+c_1c_7\\
&\leq \frac{1}{2}\|\nabla u\|_{L^2}^2+ \eta (\|\nabla \Delta v\|_{L^2}^2+\|\nabla \Delta w\|_{L^2}^2)+c_8.
\end{split}
\end{equation}
Substituting \eqref{L2e-6} into \eqref{L2e-2} gives
\begin{equation}\label{L2e-7}
\begin{split}
\frac{d}{dt}\int_\Omega u^2+\int_\Omega |\nabla u|^2\leq 2\eta (\|\nabla \Delta v\|_{L^2}^2+\|\nabla \Delta w\|_{L^2}^2)+c_{9}.
\end{split}
\end{equation}
Differentiating  the second equation of system \eqref{1-1} once, and multiplying the result by $-\nabla \Delta v$, and then we integrate the product in $\Omega$ to obtain
\begin{equation*}
\begin{split}
&\frac{1}{2}\frac{d}{dt}\int_\Omega |\Delta v|^2+D_1\int_\Omega |\nabla \Delta v|^2+\beta\int_\Omega |\Delta v|^2\\
&=-\alpha\int_\Omega \nabla \Delta v\cdot\nabla u\\
&\leq \frac{D_1}{2} \|\nabla \Delta v\|_{L^2}^2+\frac{\alpha^2}{2D_1}\|\nabla u\|_{L^2}^2,
\end{split}
\end{equation*}
which yields
\begin{equation}\label{L2e-8}
\frac{d}{dt}\int_\Omega |\Delta v|^2+D_1\int_\Omega|\nabla\Delta v|^2+2\beta\int_\Omega |\Delta v|^2
\leq \frac{\alpha^2}{D_1}\|\nabla u\|_{L^2}^2.
\end{equation}
Similarly, we have the following estimates for $w$:
\begin{equation}\label{L2e-9}
\begin{split}
\frac{d}{dt}\int_\Omega |\Delta w|^2+D_2\int_\Omega |\nabla \Delta w|^2+2\delta\int_\Omega |\Delta w|^2\leq \frac{\gamma^2}{D_2}\|\nabla u\|_{L^2}^2.
\end{split}
\end{equation}
Letting $\rho=\frac{\alpha^2D_2+\gamma^2 D_1}{D_1 D_2}$, and multiplying \eqref{L2e-7} by $2\rho$, then adding it with \eqref{L2e-8} and \eqref{L2e-9}, we end up with
\begin{equation}\label{L2e-10}
\begin{split}
&\frac{d}{dt}\left(2\rho\|u\|_{L^2}^2+\|\Delta v\|_{L^2}^2+\|\Delta w\|_{L^2}^2\right)+\rho\|\nabla u\|_{L^2}^2\\
&\ \ \ \ +D_1\|\nabla\Delta v\|_{L^2}^2+D_2\|\nabla \Delta w\|_{L^2}^2+2\beta\|\Delta v\|_{L^2}^2+2\delta\|\Delta w\|_{L^2}^2\\
&\leq4\rho \eta\cdot (\|\nabla \Delta v\|_{L^2}^2+\|\nabla \Delta w\|_{L^2}^2)+c_{10}.
\end{split}
\end{equation}
Letting $\eta$ small such that $4\rho\eta\leq \min\{D_1,D_2\}$, one has
\begin{equation}\label{L2e-10}
\begin{split}
&\frac{d}{dt}\left(2\rho\|u\|_{L^2}^2+\|\Delta v\|_{L^2}^2+\|\Delta w\|_{L^2}^2\right)+\rho\|\nabla u\|_{L^2}^2\\
&\ \ \ \ +2\beta\|\Delta v\|_{L^2}^2+2\delta\|\Delta w\|_{L^2}^2\leq c_{10}.
\end{split}
\end{equation}
On the other hand, using the Gagliardo-Nirenberg inequality and \eqref{L1-u}, we can show that
\begin{equation}\label{L2e-11}
\| u\|_{L^2}^2\leq c_{11}\left(\|\nabla u\|_{L^2}\|u\|_{L^1}+\|u\|_{L^1}^2\right)\leq \frac{1}{2}\|\nabla u\|_{L^2}^2+c_{12}.
\end{equation}
Substituting \eqref{L2e-11} into \eqref{L2e-10} and letting $y(t):=2\rho\|u\|_{L^2}^2+\|\Delta v\|_{L^2}^2+\|\Delta w\|_{L^2}^2$, we can find two positive constants $c_{13}$ and $c_{14}$ such that
\begin{equation*}
y'(t)+c_{13}y(t)\leq c_{14},
\end{equation*}
which, along with Gronwall's inequality gives  \eqref{L2e-1}.
\end{proof}
Next, we will show  the existence of global classical solutions.
\begin{lemma} \label{B11}
Suppose that the conditions in Lemma \ref{Lc1} hold. Then  the problem  \eqref{1-1} has a unique global classical  solution
 $(u,v,w)\in [C^0([0,\infty)\times \bar{\Omega})\cap C^{2,1}((0,\infty)\times\bar{\Omega})]^3$ satisfying \eqref{LT-12}.
\end{lemma}
\begin{proof} From  Lemma \ref{L2e}, we know that  there exists a constant $c_1>0$ such that $\|u(\cdot,t)\|_{L^2}\leq c_1$. Noting $n=2$ and using  Lemma \ref{BC}, one has
\begin{equation*}
\|u(\cdot,t)\|_{L^\infty}\leq c_2,
\end{equation*}
which together with the local existence results in Lemma \ref{LS} completes the proof of this lemma.
\end{proof}
\subsection{Convergence }\label{s2}
In this subsection, we will show the convergence of solutions.
% When $D_1=D_2$ and $\beta=\delta$,  it has been proved in \cite{TW-M3AS-2013} that the unique solution $(u,v,w)$ will exponentially converge to the constant state $(\bar{u}_0,\frac{\alpha}{\beta}\bar{u}_0,\frac{\gamma}{\delta}\bar{u}_0)$. Hence in the following, we will focus on studying the convergence of solutions in  the case of $D_1\neq D_2$ or $\beta\neq \delta$.
\begin{lemma}\label{K*}
 Let $(u,v,w)$ be the solution of system \eqref{1-1} satisfying \eqref{LT-12} and \eqref{Lc1-13}. Then  one has
\begin{equation}\label{K1}
\|u(\cdot,t)-\bar{u}_0\|_{L^\infty}\to 0 \ \ \mathrm{as}\ \  t\to\infty.
\end{equation}
\end{lemma}
\begin{proof}
The combination of \eqref{LT-12} and \eqref{Lc1-13} implies that there exist a constant $c_1>0$ such that
\begin{equation}\label{K2}
\int_0^\infty\|\nabla u\|_{L^2}^2\leq c_1.
\end{equation}
Noting the conservation of cell mass and  using  the  Poincar\'{e} inequality, we  will derive
\begin{equation}\label{K3}
\|u(\cdot,t)-\bar{u}_0\|_{L^2}^2=\|u(\cdot,t)-\bar{u}\|_{L^2}^2\leq c_2\|\nabla u\|_{L^2}^2.
\end{equation}
Combining \eqref{K2} and \eqref{K3}, one can find a constant $c_3>0$ such that
\begin{equation}\label{K4}
\int_0^\infty \|u(\cdot,t)-\bar{u}_0\|_{L^2}^2\leq c_3.
\end{equation}
Motivated by the ideas  in \cite[Lemma 3.10]{Tao-Winkler-SIMA-2015}, we next show \eqref{K4} implies \eqref{K1}.
Indeed, if one can show that
\begin{equation}\label{K5}
\|u(\cdot,t)-\bar{u}_0\|_{C^0}\to 0,\ \  \mathrm{as} \ \ \ t\to\infty,
\end{equation}
then  \eqref{K1} follows directly.
 We shall show \eqref{K5} by the argument of contradiction.
Suppose that  \eqref{K5} is wrong, then for some constant $c_4>0$,  there exist some sequences $(x_j)_{j\in\mathbb{N}}\subset\Omega$  and $(t_j)_{j\in\mathbb{N}}\subset(0,\infty)$  satisfying $t_j\to\infty$ as $j\to\infty$ such that
\begin{equation*}
|u(x_j,t_j)-\bar{u}_0|\geq c_4, \ \ \ \mathrm{for\ all} \ \ j\in\mathbb{N}.
\end{equation*}
From Lemma \ref{BC}, we know $u-\bar{u}_0$ is uniformly continuous in $\Omega\times(1,\infty)$. Then there exist $r>0$ and $T_1>0$ such  than for any $j\in\mathbb{N}$,
\begin{equation}\label{K6}
|u(x,t)-\bar{u}_0|\geq \frac{c_4}{2} \ \ \ \mathrm{for \ \ all}\ \ x\in B_r(x_j)\cap \Omega \ \mathrm{and}\ \  t\in(t_j,t_j+T_1).
\end{equation}
 Because of the smoothness of $\partial\Omega$, we can get a constant $c_5>0$ such that
\begin{equation}\label{K7}
|B_r(x_j)\cap\Omega|\geq c_5, \ \ \mathrm{for \  all} \  x_j\in \Omega.
\end{equation}
Using \eqref{K6} and \eqref{K7},   for all $j\in \mathbb{N}$, we have
\begin{equation}\label{K8}
\begin{split}
\int_{t_j}^{t_j+T_1}\int_\Omega |u(x,t)-\bar{u}_0|^2 dxdt
&\geq \int_{t_j}^{t_j+T_1}\int_{B_r(x_j)\cap\Omega} |u(x,t)-\bar{u}_0|^2 dxdt\\
&\geq \int_{t_j}^{t_j+T_1} |B_r(x_j)\cap\Omega|\cdot\left(\frac{c_4}{2}\right)^2dt\\
&\geq \frac{c_4^2c_5 T_1}{4}.
\end{split}
\end{equation}
However, by the fact $t_j\to\infty$ as $j\to\infty$, we have from \eqref{K4} that
\begin{equation*}
\int_{t_j}^{t_j+T_1}\int_\Omega (u(x,t)-\bar{u}_0)^2dxdt\leq \int_{t_j}^\infty\int_\Omega (u(x,t)-\bar{u}_0)^2dxdt\to  0, \mathrm{as}\ \ j\to\infty,
\end{equation*}
which contradicts \eqref{K8}. Hence \eqref{K5} holds by the argument of contradiction. Thus the proof of Lemma \ref{K*} is completed.
\end{proof}
Next, we will show the convergence of $v$ and $w$ by the comparison principle.
\begin{lemma}\label{k2}
Let the conditions in Lemma \ref{K*} hold. Then it holds that
\begin{equation*}\label{k2-1}
\|v(\cdot,t)-\frac{\alpha}{\beta}\bar{u}_0\|_{L^\infty}\to 0,\ \  \mathrm{as}\ \ t\to\infty,
\end{equation*}
and
\begin{equation*}\label{k2-2}
\|w(\cdot,t)-\frac{\gamma}{\delta}\bar{u}_0\|_{L^\infty}\to 0,\ \  \mathrm{as}\ \ t\to\infty.
\end{equation*}
\end{lemma}
\begin{proof}
Let $\phi(x,t)=v(x,t)-\frac{\alpha}{\beta}\bar{u}_0$. Then from the second equation of \eqref{1-1}, one has
\begin{equation}\label{k2-3}
\begin{cases}
\phi_t- D_1\Delta \phi+\beta \phi=\alpha (u-\bar{u}_0), &x\in\Omega, t>0,\\
\frac{\partial\phi}{\partial \nu}=0,&x\in\partial \Omega, t>0,\\
\phi(x,0)=\phi_0(x)=v_0(x)-\frac{\alpha}{\beta}\bar{u}_0, &x\in\Omega.
\end{cases}
\end{equation}
%If $\tau_1=0$, then the application of elliptic maximum principle gives
%\begin{equation}\label{k2-4}
%\Big\|v(\cdot,t)-\frac{\alpha}{\beta}\bar{u}_0\Big\|_{L^\infty}=\|\phi(\cdot,t)\|_{L^\infty}\leq \frac{\alpha}{\beta}\|u(\cdot,t)-\bar{u}_0\|_{L^\infty}.
%\end{equation}
%Combining \eqref{k2-4} and \eqref{K1} gives \eqref{k2-1}.
%Similarly when $\tau_1=0$, we can obtain
%\begin{equation}\label{k2-4*}
%\Big\|v(\cdot,t)-\frac{\alpha}{\beta}\bar{u}_0\Big\|_{L^\infty}\leq \frac{\alpha}{\beta}\|u(\cdot,t)-\bar{u}_0\|_{L^\infty}.
%\end{equation}
%Next,  we will show \eqref{k2-2} holds.
 %Let $\psi(x,t)=v(x,t)-\frac{\alpha}{\beta}\bar{u}_0$, then if $\tau_1=0$ we can use the similar argument as above to show that
 %\begin{equation}\label{k2-5}
%\Big\|v(\cdot,t)-\frac{\alpha}{\beta}\bar{u}_0\Big\|_{L^\infty}= \|\psi(\cdot,t)\|_{L^\infty}\leq \frac{\alpha}{\beta}\|u(\cdot,t)-\bar{u}_0\|_{L^\infty}.
 %\end{equation}
%If $\tau_1=1$,   then from the second equation of \eqref{ARKS-1}, one has
%\begin{equation}\label{k2-6}
%\begin{cases}
%\psi_t-\Delta \psi+\beta \psi=\alpha (u-\bar{u}_0), &x\in\Omega, t>0,\\
%\frac{\partial\psi}{\partial \nu}=0,&x\in\partial \Omega, t>0,\\
%\psi(x,0)=v_0(x)-\frac{\alpha}{\beta}\bar{u}_0:=\psi_0(x), &x\in\Omega.
%\end{cases}
%\end{equation}
Let $\phi^*(t)$ be the solution of ODE problem
\begin{equation}\label{k2-7}
\begin{cases}
\phi_t^*(t)+\beta\phi^*(t)=\alpha \|u-\bar{u}_0\|_{L^\infty},t>0,\\
\phi^*(0)=\|\phi_0\|_{L^\infty}.
\end{cases}
\end{equation}
The application of the comparison principle show that $\phi^*(t)$ is a super-solution of problem \eqref{k2-3} and satisfies
\begin{equation*}\label{k2-8}
\phi(x,t)\leq \phi^*(t) \ \ \mathrm{for\ all} \ \ x\in\Omega, t>0.
\end{equation*}
Similarly, we can prove that $\phi(x,t)\geq -\phi^*(t)$ for all $x\in\Omega,t>0$. Hence, one has
\begin{equation}\label{k2-9}
|\phi(x,t)|\leq \phi^*(t)  \ \ \mathrm{for\ all} \ \ x\in\Omega, t>0.
\end{equation}
On the other hand, using the fact $\|u(\cdot,t)-\bar{u}_0\|_{L^\infty}\to 0$ as $t\to\infty$ and  from \eqref{k2-7}  we have
\begin{equation*}\label{k2-10}
\phi^*(t)\to 0 \ \ \mathrm{as}\ \  t\to\infty,
\end{equation*}
which combined with  \eqref{k2-9}   gives
\begin{equation}\label{k2-11}
\begin{split}
\|v(\cdot,t)-\frac{\alpha}{\beta}\bar{u}_0\|_{L^\infty}
& =\|\phi(\cdot,t)\|_{L^\infty}\leq \phi^*(t)  \to 0 \ \ \mathrm{as}\ \ t\to\infty.
\end{split}
\end{equation}
Similar arguments applied to the third equation of system \eqref{1-1} yield
\begin{equation}\label{k2-11*}
\|w(\cdot,t)-\frac{\gamma}{\delta}\bar{u}_0\|_{L^\infty} \to 0,\ \  \mathrm{as}\ \  t\to\infty,
\end{equation}
which completes the proof of Lemma \ref{k2}.
\end{proof}
%\begin{remark} We should point out that the results in Lemma \ref{K*} and Lemma \ref{k2} are independent of the value of $\tau_1$ and $\tau_2$.
%\end{remark}
\subsection{Decay rate}It is shown in section \ref{s2} that $(u,v,w)\to$  $(\bar{u}_0,\frac{\alpha}{\beta}\bar{u}_0,\frac{\gamma}{\delta}\bar{u}_0)$  as $t\to\infty$ under  the condition \eqref{LT-31}. Below, we will further show the convergence rate is exponential if $\frac{\xi\gamma}{\chi\alpha}>  \max\Big\{\frac{\beta}{\delta},\frac{\delta}{\beta}\Big\}$.
%First, we prove a basic lemma.
%\begin{lemma} \label{k4}
%Suppose that $(u,v,w)$ is the solution of system \eqref{1-1}. If $\|u\|_{L^\infty}\leq C$,  then there exists a constant $k>0$ such that
%\begin{equation}\label{k4-1}
%\int_\Omega u\ln \frac{u}{\bar{u}}\leq k\int_\Omega \frac{|\nabla u|^2}{u}.
%\end{equation}
%\end{lemma}
%\begin{proof}
%From Lemma \ref{II}, one has
%\begin{equation}\label{k4-2}
%\int_\Omega u\ln \frac{u}{\bar{u}}\leq \frac{1}{2\bar{u}}\|u-\bar{u}\|_{L^2}^2.
%\end{equation}
%Using \eqref{K3} and and the fact $\|u\|_{L^\infty}\leq c_1$, one can derive that
%\begin{equation*}\label{k4-3}
%\begin{split}
%\frac{1}{2\bar{u}}\|u-\bar{u}\|_{L^2}^2\leq c_2\|\nabla u\|_{L^2}^2
%\leq c_2\|u\|_{L^\infty}\int_\Omega \frac{|\nabla u|^2}{u}\leq c_1c_2\int_\Omega \frac{|\nabla u|^2}{u},
%\end{split}
%\end{equation*}
%which combines with \eqref{k4-2} gives \eqref{k4-1} by choosing $k=c_1c_2$.
%\end{proof}
\begin{lemma} \label{FC}
Suppose that the conditions in Lemma \ref{Lc1} hold. If  $\frac{\xi\gamma}{\chi\alpha}> \max\Big\{\frac{\beta}{\delta},\frac{\delta}{\beta}\Big\}$, then there exist two constants $C>0$ and $\lambda>0$ such that
\begin{equation}\label{FC-1}
\|u(\cdot,t)-\bar{u}_0\|_{L^1}\leq C e^{-\lambda t} \ \ \  \mathrm{for\ all}\ \  t>0.
\end{equation}
\end{lemma}
\begin{proof}
%From \eqref{Lc-1},  we derive that $E(u,v,w)$ is nonnegative definite and satisfies
%\begin{equation}
 %E(u,v,w)
%\end{equation}
%if  $\frac{\xi\gamma}{\chi\alpha}\geq \max\Big\{\frac{\beta}{\delta},\frac{\delta}{\beta}\Big\}$.
%From \eqref{Lc-1} and \eqref{Lc-2}, we know that $E(u,v,w)$ and $F(u,v,w)$ are nonnegative definite
The nonnegativity of  $E(u,v,w)$ and $F(u,v,w)$ has been proved in Lemma \ref{La} under the condition \eqref{LT-31}.  Next, we show that if  $\frac{\xi\gamma}{\chi\alpha}>\max\Big\{\frac{\beta}{\delta},\frac{\delta}{\beta}\Big\}$, there exists a constant $\mu>0$ which will be chosen later such that
\begin{equation}\label{FC-2*}
E(u,v,w)\leq \mu F(u,v,w).
\end{equation}
In fact, using the definition of $E(u,v,w)$ and $F(u,v,w)$ in \eqref{La-3} and \eqref{La-4}, respectively,  we derive that
\begin{equation}\label{D}
\begin{split}
\mathcal {D}(u,v,w)
&=\mu F(u,v,w)-E(u,v,w)\\
= &\frac{\theta_1}{2\xi\chi}\left(\mu \int_\Omega \frac{|\nabla u|^2}{u}-\int_\Omega u\ln \frac{ u}{\bar{u}}\right)+\mathcal{D}_1(u,v,w)+\mathcal{D}_2(u,v,w)\\
\end{split}
\end{equation}
where
\begin{equation*}\label{D1}
\begin{split}
&\mathcal{D}_1(u,v,w)=\mu\left(\frac{\theta_2D_1}{2\xi\alpha}\int_\Omega |\Delta v|^2+\frac{\theta_2D_2}{2\gamma\chi}\int_\Omega |\Delta w|^2-(D_1+D_2)\int_\Omega \Delta w \cdot \Delta v\right)
\end{split}
\end{equation*}
and
\begin{equation*}\label{D2}
\begin{split}
\mathcal{D}_2(u,v,w)=&\frac{\theta_2}{2\xi\alpha}(\beta \mu-\frac{1}{2})\int_\Omega |\nabla v|^2
+\frac{\theta_2}{2\gamma\chi}(\delta \mu-\frac{1}{2})\int_\Omega |\nabla w|^2\\
&+[1-\mu(\beta+\delta)]\int_\Omega \nabla w\cdot\nabla v.
\end{split}
\end{equation*}
To show the nonnegativity of $\mathcal{D}(u,v,w)$, we first show the nonnegativity of first term on the right hand of \eqref{D}.  From Lemma \ref{II}, we have
\begin{equation}\label{k4-2}
\int_\Omega u\ln \frac{u}{\bar{u}}\leq \frac{1}{\bar{u}}\|u-\bar{u}\|_{L^2}^2.
\end{equation}
On the other hand, using \eqref{K3} and  the fact $\|u\|_{L^\infty}\leq c_1$, one derives
\begin{equation*}\label{k4-3}
\begin{split}
\frac{1}{\bar{u}}\|u-\bar{u}\|_{L^2}^2\leq c_2\|\nabla u\|_{L^2}^2
\leq c_2\|u\|_{L^\infty}\int_\Omega \frac{|\nabla u|^2}{u}\leq c_1c_2\int_\Omega \frac{|\nabla u|^2}{u},
\end{split}
\end{equation*}
which combined with \eqref{k4-2} gives
\begin{equation*}
\int_\Omega u\ln \frac{u}{\bar{u}}\leq c_1c_2\int_\Omega \frac{|\nabla u|^2}{u}=\mu_1 \int_\Omega \frac{|\nabla u|^2}{u},
\end{equation*}
where $\mu_1= c_1c_2$. Hence, we can  choose $\mu\geq \mu_1$ such that the first term on the right hand of \eqref{D} is nonnegative.

Next, we will show the nonnegativity of $\mathcal{D}_1(u,v,w)$. In fact, we  can rewrite $\mathcal{D}_1(u,v,w)$ as
\begin{equation*}
\mathcal{D}_1(u,v,w)=\mu \int_\Omega \Theta_2^T A_2\Theta_2,
\end{equation*}
where  $ A_2$ and $\Theta_2$ are defined in \eqref{A2-A3}. {\color{black}The condition \eqref{LT-31} gives $\frac{\xi\gamma}{\chi\alpha}\geq  \max\Big\{\frac{D_1}{D_2},\frac{D_2}{D_1}\Big\}$}. Then hence  the matrix $A_2$  is nonnegative definite and hence $\mathcal {D}_1(u,v,w)\geq 0$ for any $\mu>0$.

Similarly, to show the nonnegativity of $\mathcal{D}_2(u,v,w)$, we  rewrite it as
\begin{equation*}
\mathcal{D}_2(u,v,w)=\int_\Omega \Theta_1^T A_4 \Theta_1 ,
\end{equation*}
where
$$\ \Theta_1=\begin{bmatrix}\nabla v\\[1mm]\nabla w \end{bmatrix} \ \ \mathrm{and} \ \
A_4=\begin{bmatrix} \frac{\theta_2}{2\xi\alpha}(\beta \mu-\frac{1}{2})& \frac{1-\mu(\beta+\delta)}{2} \\[1mm]\frac{1-\mu(\beta+\delta)}{2} &\frac{\theta_2}{2\gamma\chi}(\delta \mu-\frac{1}{2})\end{bmatrix}.$$
Using the matrix analysis, we know that $A_4$ is nonnegative  definite if $\mu>\mu_2:=\max\{\frac{1}{2\beta},\frac{1}{2\beta}\}$ and
\begin{equation}\label{FC-6}
[4\theta_2^2\beta\delta-(\theta_2^2-\theta_1^2) (\beta+\delta)^2]\mu^2-2(\beta+\delta)\theta_1^2\mu+\theta_1^2\geq 0.
\end{equation}
Since $\frac{\xi\gamma}{\chi\alpha}>  \max\Big\{\frac{\beta}{\delta},\frac{\delta}{\beta}\Big\}$, one has
\begin{equation*}
\begin{split}
4\theta_2^2\beta\delta-(\theta_2^2-\theta_1^2) (\beta+\delta)^2= 4(\xi\gamma\beta-\chi\alpha\delta)(\xi\gamma\delta-\chi\alpha\beta)>0.
\end{split}
\end{equation*}
Hence \eqref{FC-6} holds if $\frac{\xi\gamma}{\chi\alpha}>  \max\Big\{\frac{\beta}{\delta},\frac{\delta}{\beta}\Big\}$ and
\begin{equation*}
\begin{split}
\mu>\mu_3:=\max\Big\{\frac{\theta_1}{2(\xi\gamma\delta-\chi\alpha\beta)},\frac{\theta_1}{2(\xi\gamma\beta-\chi\alpha\delta)}\Big\}.
\end{split}
\end{equation*}
Then if $\mu>\max\{\mu_2,\mu_3\}$, the function $\mathcal{D}_2(u,v,w)$ is nonnegative. Hence, choosing $\mu>\max\{\mu_1,\mu_2,\mu_3\}$, the function $\mathcal{D}(u,v,w)$ is nonnegative and  \eqref{FC-2*} holds.

Substituting \eqref{FC-2*} into  \eqref{La-2}, we have
\begin{equation*}\label{FC-8}
\frac{d}{dt} E(u,v,w)+\frac{1}{\mu} E(u,v,w)\leq 0,
\end{equation*}
which implies
\begin{equation}\label{FC-9}
E(u,v,w)\leq c_3 e^{-\frac{1}{\mu} t}.
\end{equation}
On the other hand, from \eqref{PE*}, we have
\begin{equation*}\label{PE}
 \frac{\theta_1}{2\xi\chi}\int_\Omega u\ln\frac{u}{\bar{u}}\leq E(u,v,w),
\end{equation*}
which along with  \eqref{FC-9} and Lemma \ref{II} gives
\begin{equation*}
\|u(\cdot,t)-\bar{u}_0\|_{L^1}^2=\|u(\cdot,t)-\bar{u}\|_{L^1}^2\leq \frac{4\xi\chi\bar{u}c_3}
{\theta_1} e^{-\frac{1}{\mu} t}.
\end{equation*}
This yields \eqref{FC-1} and concludes the proof.
\end{proof}
Next, we will derive  the decay rate of solutions in $L^\infty$-norm based on the
 decay rate of $\|u(\cdot,t)-\bar{u}_0\|_{L^1}$.
 \begin{lemma}\label{L1D} Let $(u,v,w)$ be the global classical solution of system \eqref{1-1}. Suppose that there exists two positive constant $C,\lambda$ such that
\begin{equation}\label{L1u*}
\|u(\cdot,t)-\bar{u}_0\|_{L^1}\leq C e^{-\lambda t},
\end{equation}
then the solution $(u,v,w)$ will exponentially decay to $(\bar{u}_0,\frac{\alpha}{\beta} \bar{u}_0,\frac{\gamma}{\delta}\bar{u}_0)$ with $L^\infty$-norm as $t\to\infty$.
\end{lemma}
\begin{proof}
%Using \eqref{A10*-1},  one can readily get a constant $c_{1}>0$ (e.g. see \cite[Lemma 3.14]{Tao-Winkler-SIMA-2015}) such that
%\begin{equation*}
%\|\nabla u\|_{L^\infty} \leq c_{1}, ~\mathrm{for~all}~t>1，
%\end{equation*}
%which, along with the Gagliardo-Nirenberg inequality, gives
%\begin{equation*}
%\begin{split}
%\|u-\bar{u}_0\|_{L^\infty}
%&\leq c_2\left( \|\nabla u\|_{L^\infty}^\frac{n}{n+1}\|u-\bar{u}_0\|_{L^1}^\frac{1}{n+1}+\|u-\bar{u}_0\|_{L^1}\right)\\
%&\leq c_3 e^{-c_3t} ~\mathrm{for~all}~t>1.
%\end{split}
%\end{equation*}
With \eqref{L1u*} in hand, we can use the  Moser-Alikakos iteration procedure as
in \cite{TW-M3AS-2013} or the semigroup estimate method in \cite{LM-NAR-2016} to obtain
\begin{equation*}
\|u-\bar{u}_0\|_{L^\infty}\leq c_1e^{-c_1 t}.
\end{equation*}
Then applying the  comparison principle as in \cite{TW-M3AS-2013}, one can show that there exists a constant $c_2>0$ such that
\begin{equation*}
\|v-\frac{\alpha}{\beta}\bar{u}_0\|_{L^\infty}+\|w-\frac{\gamma}{\delta}\bar{u}_0\|_{L^\infty}\leq c_2 e^{-c_2 t}.
\end{equation*}
Then the proof of this lemma is completed.
\end{proof}
%\begin{remark}We should point out that the method used in the proof of  convergence is independent of space dimensions.
%\end{remark}
\subsection{Proof of Theorem \ref{LT-3}}
Under the condition \eqref{LT-31},  we show the boundedness of solution for system \eqref{1-1} with $D_1\neq D_2$  in Lemma \ref{B11}, which implies  there exists a constant $c_1>0$ such that $\|u(\cdot,t)\|_{L^\infty}\leq c_1$. Moreover, from Lemma \ref{Lc1}, one has a constant $c_2>0$ such that
\begin{equation*}
\int_0^t\int_\Omega \frac{|\nabla u|^2}{u}\leq c_2,
\end{equation*}
which together with the fact $\|u(\cdot,t)\|_{L^\infty}\leq c_1$ implies
 $\|u-\bar{u}_0\|_{L^\infty} \to 0$ as $t\to\infty$ as shown in Lemma \ref{K*}. Then using the comparison principle   for parabolic equations, from the second and third equations of system \eqref{1-1}, we  show that  the solution $(v,w)$  converges to $(\frac{\alpha}{\beta}\bar{u}_0,\frac{\gamma}{\beta}\bar{u}_0)$ as $t\to\infty$ in  Lemma \ref{k2}. Moreover, if  $\xi\gamma>\chi\alpha \max\Big\{\frac{\beta}{\delta},\frac{\delta}{\beta}\Big\}$, then using Lemma \ref{FC}, we can obtain
 \begin{equation*}
 \|u(\cdot,t)-\bar{u}_0\|_{L^1}\leq c_3e^{-\lambda t},
 \end{equation*}
 which, along with Lemma \ref{L1D}, gives the exponential decay rate as shown in Theorem \ref{LT-3}. Then Theorem \ref{LT-3} is proved.

\bigbreak

\noindent \textbf{Acknowledgement}.  The authors would like to thank the referee for his/her comments and suggestions on the improvement of the paper.
 The research of H.Y. Jin was supported  by  the NSF of China No. 11871226. The research of Z.A. Wang was supported by the Hong Kong RGC GRF grant
No. PolyU 5091/13P.

\end{document}